\documentclass[reqno]{amsart}
\usepackage{amssymb,amsmath,enumerate}
\usepackage{tikz,hyperref}

\usetikzlibrary{decorations.pathreplacing}

\tikzstyle{path}=[draw, line width=1.2, color=black]
\tikzstyle{pathlight}=[draw, line width=1, dotted, color=lightgray]

\newtheorem{theorem}{Theorem}[section]

\newtheorem{corollary}[theorem]{Corollary}
\newtheorem{lemma}[theorem]{Lemma}

\theoremstyle{definition}
\newtheorem{definition}[theorem]{Definition}
\newtheorem*{definition*}{Definition}

\def\L{\mathcal{L}}

\def\s{\circledast}
\def\E{\mathsf{E}}
\def\N{\mathsf{N}}
\def\EE{\mathsf{ee}}
\def\NN{\mathsf{nn}}
\def\NE{\mathsf{ne}}
\def\EN{\mathsf{en}}

\title{Bounce statistics for rational lattice paths}
\author{Daniel Birmajer}
\address{Department of Mathematics\\ Nazareth College\\ 4245 East Ave.\\ Rochester, NY 14618}
\author{Juan B. Gil}
\address{Penn State Altoona\\ 3000 Ivyside Park\\ Altoona, PA 16601}
\author{Michael D. Weiner}


\begin{document}
\maketitle

\begin{abstract}
Given two relatively prime positive integers $\alpha$ and $\beta$, we consider simple lattice paths (with unit East and unit North steps) from $(0,0)$ to $(\alpha k,\beta k)$, and enumerate them by their left and right bounces with respect to the line $y=\frac{\beta}{\alpha} x$. We give the corresponding multivariate generating functions for all such paths as well as for subclasses of paths that start and end with a prescribed step. For illustration purposes, we discuss the case $\beta=1$ and express some of our functions in terms of the Fuss-Catalan generating function $c_\alpha(x)$.
\end{abstract}

\section{Introduction}

Motived by recent work of Pan and Remmel \cite{PanRe16}, we consider the class $\L_{\beta/\alpha}(k)$ of simple lattice paths (with east `$\E$' and north `$\N$' unit steps) from $(0,0)$ to $(\alpha k,\beta k)$, where $k$, $\alpha$, and $\beta$ are positive integers with $\gcd(\alpha, \beta) = 1$. An $\mathsf{AB}$-path is such a lattice path that starts with an $\mathsf{A}$-step and ends with a $\mathsf{B}$-step with $\mathsf{A,B}\in\{\E,\N\}$. 

\begin{figure}[ht]
 \begin{tikzpicture}[scale=0.45]
 \begin{scope}
 \draw [step=1,gray!60] (0,0) grid (4,6);
 \draw [pathlight] (0,0) -- (4,6);
 \draw [path] (0,0) -- (2,0) -- (2,3) -- (3,3) -- (3,6) -- (4,6);
 \node[below=1pt] at (2,0) {\small $\E\E$-path};
 \end{scope}

 \begin{scope}[xshift=75mm]
 \draw [step=1,gray!60] (0,0) grid (4,6);
 \draw [pathlight] (0,0) -- (4,6);
 \draw [path] (0,0) -- (0,2) -- (2,2) -- (3,2) -- (3,4) -- (4,4) -- (4,6);
 \node[below=1pt] at (2,0) {\small $\N\N$-path};
 \end{scope}

 \begin{scope}[xshift=150mm]
 \draw [step=1,gray!60] (0,0) grid (4,6);
 \draw [pathlight] (0,0) -- (4,6);
 \draw [path] (0,0) -- (1,0) -- (1,3) -- (2,3) -- (2,5) -- (4,5) -- (4,6);
 \node[below=1pt] at (2,0) {\small $\E\N$-path};
 \end{scope}
 \end{tikzpicture}
\caption{Examples of lattice paths in $\L_{3/2}(2)$.}
\end{figure}
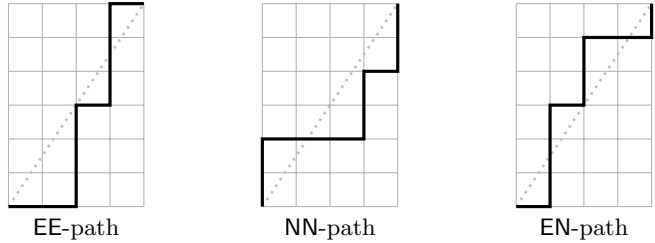

It is easy to see that there are $\binom{(\alpha+\beta)k}{\alpha k}$ lattice paths in $\L_{\beta/\alpha}(k)$, so we get the generating function
\begin{equation*}
   g(x)=\sum_{k=1}^{\infty}\binom{(\alpha+\beta)k}{\alpha k}x^k,
\end{equation*}
and similarly we have
\begin{equation*}
g_{\EE}(x)=\sum_{k=1}^{\infty}\tbinom{(\alpha+\beta)k-2}{\alpha k-2}x^k, \;\; 
g_{\EN}(x)=\sum_{k=1}^{\infty}\tbinom{(\alpha+\beta)k-2}{\alpha k-1}x^k, \;\; 
g_{\NN}(x)=\sum_{k=1}^{\infty}\tbinom{(\alpha+\beta)k-2}{\alpha k}x^k,
\end{equation*}
where  $g_{\mathsf{ab}}(x)$ denotes the generating function enumerating $\mathsf{AB}$-paths. Note that $g_{\NE}(x)=g_{\EN}(x)$, and
\begin{equation}\label{eq:gRelations}
 \beta \big(g_{\EE}(x) + g_{\EN}(x)\big) = \alpha \big(g_{\NN}(x) + g_{\EN}(x)\big).
\end{equation}
In addition to counting the number of lattice paths in $\L_{\beta/\alpha}(k)$ that start with an $\E$-step, the function $g_{\E\s}(x) = g_{\EE}(x) + g_{\EN}(x)$ also enumerates ordered partitions (weak compositions) of $\beta k$ into $\alpha k$ parts. There is a similar interpretation for $g_{\N\s}(x) = g_{\NN}(x) + g_{\EN}(x)$.

\begin{definition} \label{def:GFs}
We say that a path $L$ has a ``left bounce'' (resp.\ ``right bounce'') if a vertex of the form $\E\N$ (resp.\ $\N\E$) on $L$ touches the line $y=\frac{\beta}{\alpha}x$, see Figure~\ref{fig:bounces}. A path is said to be {\em bounce-free} with respect to $y=\frac{\beta}{\alpha}x$ if it has no bounces (left or right) at that line, see Figure~\ref{fig:bounceFree}. We let $f(x)$ be the generating function enumerating bounce-free rational paths with respect to the line $y=\frac{\beta}{\alpha}x$, and let $f_{\mathsf{ab}}(x)$ denote the corresponding function for bounce-free $\mathsf{AB}$-paths. 
\end{definition}

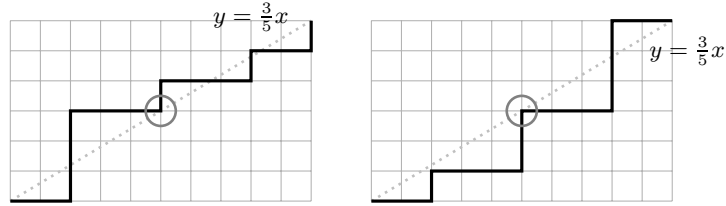
\begin{figure}[ht]
 \begin{tikzpicture}[scale=0.4]
 \begin{scope}
 \draw [step=1,gray!60] (0,0) grid (10,6);
 \draw [pathlight] (0,0) -- (10,6);
 \draw [path] (0,0) -- (2,0) -- (2,3) -- (5,3) -- (5,4) -- (8,4) -- (8,5) -- (10,5) -- (10,6);
 \draw[line width=1pt,gray] (5,3) circle (0.5);
 \node[above=4pt] at (8,5) {\small $y=\frac35 x$};
 \end{scope}
 \begin{scope}[xshift=120mm]
 \draw [step=1,gray!60] (0,0) grid (10,6);
 \draw [pathlight] (0,0) -- (10,6);
 \draw [path] (0,0) -- (2,0) -- (2,1) -- (5,1) -- (5,3) -- (8,3) -- (8,6) -- (10,6);
 \draw[line width=1pt,gray] (5,3) circle (0.5);
 \node at (10.5,5) {\small $y=\frac35 x$};
 \end{scope}
 \end{tikzpicture}
\caption{Examples of a left and a right bounce, respectively.}
\label{fig:bounces}
\end{figure}

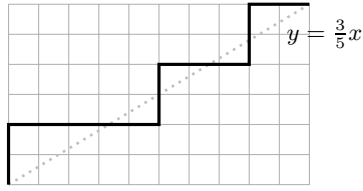
\begin{figure}[ht]
 \begin{tikzpicture}[scale=0.4]
 \draw [step=1,gray!60] (0,0) grid (10,6);
 \draw [pathlight] (0,0) -- (10,6);
 \draw [path] (0,0) -- (0,2) -- (5,2) -- (5,4) -- (8,4) -- (8,6) -- (10,6);
 \node at (10.5,5) {\small $y=\frac35 x$};
 \end{tikzpicture}
\caption{Example of a bounce-free $\N\E$-path in $\L_{3/5}(2)$.}
\label{fig:bounceFree}
\end{figure}

The main goal of this paper is to derive generating functions $G_{\mathsf{ab}}(x,s,t)$ that count $\mathsf{AB}$-paths in $\L_{\beta/\alpha}(k)$, where the coefficient of $s^\ell t^r x^k$ gives the number of lattice paths from $(0,0)$ to $(\alpha k,\beta k)$ having exactly $\ell$ left bounces and $r$ right bounces. For this purpose, we first obtain expressions for the bounce-free functions $f(x)$ and $f_{\mathsf{ab}}(x)$ in terms of the known functions $g(x)$ and $g_{\mathsf{ab}}(x)$. 

Finally, in Section~\ref{sec:beta=1} we discuss the particular case $\alpha\ge 1$, $\beta=1$ (which by symmetry is similar to the case $\alpha=1$, $\beta\ge 1$) and give expressions for the bounce-free generating functions in terms of the Fuss-Catalan function $c_{\alpha}(x)$.

\section{Bounce-free rational lattice paths}

Throughout this section we fix $\alpha, \beta\in\mathbb{N}$ such that $\gcd(\alpha,\beta)=1$. As before, we let $g_{\mathsf{ab}}(x)$ denote the generating function for the number of $\mathsf{AB}$-paths from $(0,0)$ to $(\alpha k,\beta k)$, and we let $f_{\mathsf{ab}}(x)$ be the corresponding function for the $\mathsf{AB}$-paths that are bounce-free with respect to the line $y=\frac{\beta}{\alpha}x$.

\begin{lemma}
The generating function for the number of\, $\E\N$-paths with no right bounces is given by
\begin{equation} \label{eq:lemmaOne}
  \textup{nrb}_{\EN}(x) = f_{\EN}(x)+\frac{f_{\EE}(x)f_{\NN}(x)}{1-f_{\EN}(x)} = \frac{g_{\EN}(x)}{1+g_{\EN}(x)}.
\end{equation}
This is the same function as for the number of\, $\E\N$-paths with no left bounces.
\end{lemma}

\begin{proof}
First, recall that the number of bounce-free $\E\N$-paths is given by $f_{\EN}(x) = f_{\NE}(x)$. On the other hand, any left bounce on the line $y=\frac{\beta}{\alpha}x$ breaks up the path into two path segments: an $\E\E$-path segment and an $\N\N$-path segment:

\begin{center}
\begin{tikzpicture}[scale=0.4]
 \draw [step=1,gray!60] (0,0) grid (10,6);
 \draw [pathlight] (0,0) -- (10,6);
 \draw [path] (0,0) -- (3,0) -- (3,3) -- (5,3);
 \draw [path,blue] (5,3) -- (5,4) -- (9,4) -- (9,5) -- (10,5) -- (10,6);
 \draw[fill,red] (5,3) circle (0.15);
 \node at (8,6.2) {\small $y=\frac{\beta}{\alpha} x$};
\end{tikzpicture}
\end{center}

If a lattice path $P$ has exactly $m\ge 1$ left bounces and no right bounces, we can represent $P$ as a concatenation of $m+1$ bounce-free path segments
\[ P = P_1^{\EE}\, P_2^{\NE} \cdots P_{m}^{\NE}\, P_{m+1}^{\NN}, \]
each one starting and ending at the line $y=\frac{\beta}{\alpha}x$. Here $P_i^{\mathsf{ab}}$ denotes a bounce-free path starting with an $\mathsf{A}$-step and ending with a $\mathsf{B}$-step. Therefore, the total number of paths $P$ with no right bounces and having at least one left bounce is given by the generating function
\begin{equation*}
 f_{\EE}(x)\bigg(\sum_{j=0}^{\infty}f_{\EN}(x)^j\bigg)f_{\NN}(x)=\frac{f_{\EE}(x)f_{\NN}(x)}{1-f_{\EN}(x)}.  
\end{equation*}
In conclusion, $f_{\EN}(x) + \frac{f_{\EE}(x)f_{\NN}(x)}{1-f_{\EN}(x)}$ gives the number of $\E\N$-paths that contain no right bounces.

Let us now examine the right-hand side of \eqref{eq:lemmaOne}. Recall that $g_{\EN}(x)$ is the generating function for the set $\L_{\beta/\alpha}^{\EN}$ of all $\E\N$-paths from $(0,0)$ to a point on the line $y=\frac{\beta}{\alpha} x$. For $i\ge 1$, consider the set
\[ A_i = \{\text{paths in } \L_{\beta/\alpha}^{\EN} \text{ with a right bounce at the point } (\alpha i, \beta i)\}. \]
By the inclusion-exclusion principle, the number of paths with no right bounces at the points $(\alpha i, \beta i)$ for $i=1,\dots,m$, is given by
\begin{equation}\label{eq:incl-excl}
 |\L_{\beta/\alpha}^{\EN}| - \sum |A_i| + \sum |A_i\cap A_j| - \sum |A_i\cap A_j\cap A_k| + \dots + (-1)^m |A_1\cap \dots \cap A_m|.
\end{equation}
Since the set of $\E\N$-paths with $m$ right bounces has the generating function $g_{\EN}(x)^{m+1}$, we get that the generating function for \eqref{eq:incl-excl} is $g_{\EN}(x) - g_{\EN}(x)^{2} + \dots + (-1)^m g_{\EN}(x)^{m+1}$. Taking the limit as $m\to\infty$, we get that
\[ \sum_{i=0}^\infty (-1)^i g_{\EN}(x)^{i+1} = \frac{g_{\EN}(x)}{1+g_{\EN}(x)} \]
gives the number of $\E\N$-paths that contain no right bounces. This proves \eqref{eq:lemmaOne}.
\end{proof}

\begin{lemma} 
The generating functions for the number of\, $\E\E$-paths with no right bounces and for the number of\, $\N\N$-paths with no right bounces, respectively, are given by
\begin{equation} \label{eq:lemmaTwo}
 \textup{nrb}_{\EE}(x) = \frac{f_{\EE}(x)}{1-f_{\EN}(x)} = \frac{g_{\EE}(x)}{1+g_{\EN}(x)} \;\text{ and }\;
 \textup{nrb}_{\NN}(x) = \frac{f_{\NN}(x)}{1-f_{\EN}(x)} = \frac{g_{\NN}(x)}{1+g_{\EN}(x)}.
\end{equation}
These functions also give the number of corresponding paths with no left bounces.
\end{lemma}

\begin{proof}
First of all, it is easy to argue that 
\begin{equation*}
 f_{\EE}(x) \sum_{j=0}^{\infty}f_{\NE}(x)^j =  \frac{f_{\EE}(x)}{1-f_{\NE}(x)} = \frac{f_{\EE}(x)}{1-f_{\EN}(x)}  
\end{equation*}
counts the number of $\E\E$-paths with no right bounces. 

To show that $\frac{g_{\EE}(x)}{1+g_{\EN}(x)}$ counts the same number of paths, consider the set
\[ A_i = \{\text{paths in } \L_{\beta/\alpha}^{\EE} \text{ with a right bounce at } (\alpha i, \beta i)\}. \]
As we did in the proof of the previous lemma, we apply the inclusion-exclusion principle to conclude that the number of paths with no right bounces at the points $(\alpha i, \beta i)$ for $i=1,\dots,m$, is counted by the generating function $g_{\EE}(x) - g_{\EN}(x)g_{\EE}(x) + g_{\EN}(x)^2g_{\EE}(x) - \dots + (-1)^m g_{\EN}(x)^{m}g_{\EE}(x)$, thus
\[ \sum_{i=0}^\infty (-1)^i g_{\EN}(x)^{i}g_{\EE}(x) = \frac{g_{\EE}(x)}{1+g_{\EN}(x)} \]
also counts the number of $\E\E$-paths that contain no right bounces. This proves the first of the two equations claimed in \eqref{eq:lemmaTwo}. The second equation follows by the symmetry between the pair $(\alpha,\beta)$ with $\E\E$-paths and the pair $(\beta,\alpha)$ with $\N\N$-paths.
\end{proof}

As a consequence of the previous two lemmas, we arrive at the following result.
\begin{theorem} \label{thm:bounce_free}
Let $f_{\mathsf{ab}}(x)$ be the generating function for the number of $\mathsf{AB}$-paths that are bounce-free with respect to the line $y=\frac{\beta}{\alpha}x$. Then
\begin{gather*}
 f_{\EE}(x) = \frac{g_{\EE}(x)}{(1+g_{\EN}(x))^2 - g_{\EE}(x)g_{\NN}(x)}, \quad
 f_{\NN}(x) = \frac{g_{\NN}(x)}{(1+g_{\EN}(x))^2 - g_{\EE}(x)g_{\NN}(x)}, \\[1ex]
 f_{\EN}(x) = 1 - \frac{1+ g_{\EN}(x)}{(1+g_{\EN}(x))^2 - g_{\EE}(x)g_{\NN}(x)}.
\end{gather*}
\end{theorem}

\begin{corollary} \label{cor:bounce_free}
If $f(x)$ enumerates the bounce-free lattice paths in $\L_{\beta/\alpha}(k)$, then
\begin{equation*}
 f(x) =\frac{g(x) + 2\left(g_{\EN}(x)^2 - g_{\EE}(x)g_{\NN}(x)\right)}{(1+g_{\EN}(x))^2 - g_{\EE}(x)g_{\NN}(x)}.
\end{equation*}
\end{corollary}

\section{Left-bounce and right-bounce statistics}

Let $b_{\,\ell,r}(k)$ be the number of paths in $\L_{\beta/\alpha}(k)$ having exactly $\ell$ left bounces and $r$ right bounces on the line $y=\frac{\beta}{\alpha} x$, see Figure~\ref{fig:bounces}, and let $B_{\ell,r}(x)$ be the corresponding generating function. 

Let $f_{\mathsf{A}\s}(x)$ denote the generating function for the number of bounce-free paths that start with an $\mathsf{A}$-step, and let $f_{\s\mathsf{B}}(x)$ enumerate the bounce-free paths that end with a $\mathsf{B}$-step. Thus
\begin{equation*}
 f_{\s\E}(x) = f_{\E\s}(x) = f_{\EE}(x) + f_{\NE}(x) \:\text{ and }\: f_{\N\s}(x) = f_{\s\N}(x) = f_{\NN}(x) + f_{\NE}(x).
\end{equation*}

Clearly, $B_{0,0}(x) = f(x)$, and for $\ell,r \in\mathbb{N}$, we have
\begin{equation} \label{eq:onesidedBounces}
 B_{\ell,0} = f_{\s\E}(x) f_{\NE}(x)^{\ell-1} f_{\N\s}(x) \:\text{ and }\:
 B_{0,r} = f_{\s\N}(x)f_{\EN}(x)^{r-1} f_{\E\s}(x).
\end{equation}

\begin{lemma} \label{lem:no_leftBounces}
The generating function for the number of paths in $\L_{\beta/\alpha}(k)$ with no left (or right) bounces on the line $y=\frac{\beta}{\alpha} x$ is given by
\begin{equation*}
 \sum_{\ell=0}^\infty B_{\ell,0}(x) = \sum_{r=0}^\infty B_{0,r}(x) 
 = \frac{g(x) + g_{\EN}(x)^2 - g_{\EE}(x)g_{\NN}(x)}{1+g_{\EN}(x)}.
\end{equation*}
\end{lemma}

\begin{proof}
Using equations \eqref{eq:onesidedBounces}, \eqref{eq:lemmaOne} and \eqref{eq:lemmaTwo}, we get
\begin{align} \notag
\sum_{\ell=0}^\infty B_{\ell,0}(x) 
 &= B_{0,0}(x) + \sum_{\ell=1}^\infty f_{\s\E}(x) f_{\NE}(x)^{\ell-1} f_{\N\s}(x) \\ 
 \notag
 &= f_{\EE}(x) + f_{\NN}(x) +2f_{\EN}(x) + \frac{[f_{\EE}(x) + f_{\EN}(x)][f_{\NN}(x) + f_{\EN}(x)]}{1-f_{\EN}(x)} \\
 \label{eq:f_split}
 &= \frac{g_{\EE}(x) + g_{\NN}(x) +g_{\EN}(x)}{1+g_{\EN}(x)} +  \frac{f_{\EN}(x)}{1-f_{\EN}(x)}
\end{align}
Now, by the third equation of Theorem~\ref{thm:bounce_free}, we have
\begin{equation*}
  \frac{1}{1 - f_{\EN}(x)} = \frac{(1+g_{\EN}(x))^2 - g_{\EE}(x)g_{\NN}(x)}{1+ g_{\EN}(x)},
\end{equation*}
and therefore
\begin{equation*}
  \frac{f_{\EN}(x)}{1 - f_{\EN}(x)} = \frac{1}{1 - f_{\EN}(x)} - 1
  = \frac{g_{\EN}(x) + g_{\EN}(x)^2 - g_{\EE}(x)g_{\NN}(x)}{1+ g_{\EN}(x)}.
\end{equation*}
Putting this expression into \eqref{eq:f_split} and simplifying gives the claimed identity.
\end{proof}

\begin{lemma} \label{lem:bounceStats}
Let $B_{\ell,r}(x)$ be the generating function for the number of paths in $\L_{\beta/\alpha}(k)$ having exactly $\ell$ left bounces and $r$ right bounces on the line $y=\frac{\beta}{\alpha} x$. For $\ell,r>0$, we have
\begin{align*}
  B_{\ell,r}(x) = 
  & \sum_{i=1}^{\ell-1}\binom{\ell-1}{i}\binom{r-1}{i-1} 
  f_{\s\E}(x) f_{\N\s}(x) \big(f_{\EE}(x)f_{\NN}(x)\big)^i f_{\EN}(x)^{\ell+r-2i-1} \\
  &+ \sum_{i=1}^{\ell}\binom{\ell-1}{i-1}\binom{r-1}{i-1} 
  f_{\s\E}(x)^2f_{\EE}(x)^{i-1}f_{\NN}(x)^{i}f_{\EN}(x)^{\ell+r-2i} \\
  &+ \sum_{i=1}^{\ell}\binom{\ell-1}{i-1}\binom{r-1}{i-1} 
  f_{\s\N}(x)^2f_{\EE}(x)^{i}f_{\NN}(x)^{i-1}f_{\EN}(x)^{\ell+r-2i} \\
  &+ \sum_{i=2}^{\ell+1}\binom{\ell-1}{i-2}\binom{r-1}{i-1}
  f_{\s\N}(x) f_{\E\s}(x) \big(f_{\EE}(x)f_{\NN}(x)\big)^{i-1} f_{\EN}(x)^{\ell+r-2i+1}.
\end{align*}
\end{lemma}

\begin{proof}
We adopt the standard convention $\binom{m}{n}=0$ when $n>m\ge 0$. We are interested in counting all paths $P$ having exactly $r$ right bounces and $\ell$ left bounces. The vertices of such bounces correspond to lattice points on the line $y=\frac{\beta}{\alpha}x$. We group the right bounces into segments with no left bounces, and such that there must be at least one left bounce between two groups. If there are $i$ such segments, this grouping induces a composition of $r$ into $i$ parts $r = r_1+\dots+r_i$, where $r_j$ is the number of right bounces in the $j$th group. There are $\binom{r-1}{i-1}$ such compositions.  

Assume that $P$ is a path with $i$ such groups of right bounces. The $\ell$ left bounces must then be distributed between the origin and the first right bounce, the gaps between the $i$ segments defined above, and between the last right bounce and the end of the path. There are four disjoint cases.

\smallskip\noindent
{\sc Case 1:} The first and last bounces are left bounces. In this case there are $i+1$ segments to place the left bounces, which can be done in $\binom{\ell-1}{i}$ different ways (compositions of $\ell$ into $i+1$ parts). Assume the bounces are distributed according to the compositions $r=r_1+\dots+r_i$ and $\ell=\ell_1+\dots+\ell_{i+1}$. Thus between the origin and the last right bounce of the first group, the path is of the form
\[ u_{1,0}\, \underbrace{\E|\N\, u_{1,1}\, \E|\N \cdots \E|\N\, u_{1,\ell_1-1}\, \E|\N}_{\ell_1 \text{ left bounces}}\, 
  w_1 \underbrace{\N|\E\,v_{1,1}\, \N|\E\cdots \N|\E\, v_{1,r_1-1}\, \N|{\color{red}\E}}_{r_1 \text{ right bounces}}\,, \]
where the symbol $|$ indicates a bounce and the $u_{i,j}, v_{i,j}, w_j$ represent bounce-free path segments. Since there are $\ell_1$ left bounces and $r_1$ right bounces in this group, the enumeration of such path segments can be achieved by the generating function $f_{\s\E}(x) f_{\NE}(x)^{\ell_1-1} f_{\NN}(x) f_{\EN}(x)^{r_1-1}$.

The first path segment is then followed by $i-1$ segments of the form
\begin{equation}\label{eq:midSegmentPartition}
 {\color{red}\E}\, u_{j,0}\, 
  \underbrace{\E|\N\, u_{j,1}\, \E|\N \cdots \E|\N\, u_{j,\ell_j-1}\, \E|\N}_{\ell_j \text{ left bounces}}\, w_j 
  \underbrace{\N|\E\,v_{j,1}\, \N|\E\cdots \N|\E\, v_{j,r_j-1}\, \N|{\color{red}\E}}_{r_j \text{ right bounces}}\,,  
  \quad 2\le j \le i, 
\end{equation}
where the first $\E$-step is the last step of the previous group. The number of all possible such segments can be counted using the generating function 
\begin{equation}\label{eq:midSegmentFunction}
 f_{\EE}(x) f_{\NE}(x)^{\ell_j-1} f_{\NN}(x) f_{\EN}(x)^{r_j-1}, \quad 2\le j \le i.
\end{equation}
Finally, the last group (containing only left bounces) must be of the form 
\[ {\color{red}\E}\, u_{i+1,0}\, \underbrace{\E|\N\, u_{i+1,1}\, \E|\N \cdots 
  \E|\N\, u_{i+1,\ell_{j+1}-1}\, \E|\N}_{\ell_{i+1} \text{ left bounces}}\, w_{i+1}, \]
and the corresponding generating function takes the form $f_{\EE}(x) f_{\NE}(x)^{\ell_{i+1}-1} f_{\N\s}(x)$. Multiplying all the pieces together we arrive at the generating function
\begin{equation*}
  \sum_{i=1}^{\ell-1}\binom{\ell-1}{i}\binom{r-1}{i-1} 
  f_{\s\E}(x) f_{\N\s}(x) \big(f_{\EE}(x)f_{\NN}(x)\big)^i f_{\EN}(x)^{\ell+r-2i-1}.
\end{equation*}

\smallskip\noindent
{\sc Case 2:} The first bounce of $P$ is a left bounce and its last bounce is a right bounce. In this case, there are $i$ segments to place the left bounces, which can be done in $\binom{\ell-1}{i-1}$ different ways. The partition of the path into groups is similar as in {\sc Case 1}, except that the last segment has now no left bounces. Thus the function $f_{\EE}(x) f_{\NE}(x)^{\ell_{i+1}-1} f_{\N\s}(x)$ from before should be replaced by $f_{\E\s}(x)$. Multiplying the pieces together we arrive at
\begin{equation*}
  \sum_{i=1}^{\ell}\binom{\ell-1}{i-1}\binom{r-1}{i-1} 
  f_{\s\E}(x)^2 f_{\EE}(x)^{i-1} f_{\NN}(x)^i f_{\EN}(x)^{\ell+r-2i}.
\end{equation*}

\smallskip\noindent
{\sc Case 3:} The first bounce of $P$ is a right bounce and its last bounce is a left bounce. Again, this may happen in $\binom{\ell-1}{i-1}$ different ways. With a similar argument as for the previous step, we get
\begin{equation*}
  \sum_{i=1}^{\ell}\binom{\ell-1}{i-1}\binom{r-1}{i-1} 
  f_{\s\N}(x)^2 f_{\EE}(x)^{i} f_{\NN}(x)^{i-1} f_{\EN}(x)^{\ell+r-2i}.
\end{equation*}

\smallskip\noindent
{\sc Case 4:} The first and last bounces of $P$ are right bounces. Thus there are only $i-1$ segments to place the left bounces, which can be done in $\binom{\ell-1}{i-2}$ different ways. These $i-1$ path segments containing left bounces are of the form \eqref{eq:midSegmentPartition} and can therefore be counted by \eqref{eq:midSegmentFunction}. The first segment must be of the form $w_1\, \N|\E\,v_{1,1}\, \N|\E\cdots \N|\E\, v_{1,r_1-1}\, \N|\E$ with $r_1$ right bounces, and $P$ must end with a bounce-free segment that starts with an $\E$-step. These two segments are counted by $f_{\s\N}(x) f_{\EN}(x)^{r_1-1} f_{\E\s}(x)$. Combining the information for the various pieces of $P$, we get
\begin{equation*}
  \sum_{i=2}^{\ell+1}\binom{\ell-1}{i-2}\binom{r-1}{i-1}
  f_{\s\N}(x) f_{\E\s}(x) \big(f_{\EE}(x)f_{\NN}(x)\big)^{i-1} f_{\EN}(x)^{\ell+r-2i+1}.
\end{equation*}
\smallskip
Finally, $B_{\ell,r}(x)$ is just the sum of the functions obtained above for the four disjoint cases.
\end{proof}

\begin{theorem} \label{thm:mainGF}
Let $B_{\ell,r}(x)$ be the generating function for the number of paths in $\L_{\beta/\alpha}(k)$ having exactly $\ell$ left bounces and $r$ right bounces on the line $y=\frac{\beta}{\alpha} x$. Then for the multivariate generating function $G(x,s,t) = \sum_{\ell,r\ge 0} B_{\ell,r}(x) s^\ell t^r$, we have
\begin{equation*}
  G(x,s,t) = \frac{g(x) + (2-s-t)\big(g_{\EN}(x)^2 - g_{\EE}(x)g_{\NN}(x)\big)}%
  {1 + (2-s-t)g_{\EN}(x) + (1-s)(1-t)(g_{\EN}(x)^2-g_{\EE}(x)g_{\NN}(x))}.
\end{equation*}
Moreover, if $G_{\mathsf{ab}}(x,s,t)$ denotes the restriction of $G(x,s,t)$ to $\mathsf{AB}$-paths, then
\begin{align*}
  G_{\EE}(x,s,t) &= \frac{g_{\EE}(x)}{1 + (2-s-t)g_{\EN}(x) + (1-s)(1-t)(g_{\EN}(x)^2-g_{\EE}(x)g_{\NN}(x))}, \\
  G_{\NN}(x,s,t) &= \frac{g_{\NN}(x)}{1 + (2-s-t)g_{\EN}(x) + (1-s)(1-t)(g_{\EN}(x)^2-g_{\EE}(x)g_{\NN}(x))}, \\
  G_{\EN}(x,s,t) &= \frac{g_{\EN}(x) + (1-s)\big(g_{\EN}(x)^2 - g_{\EE}(x)g_{\NN}(x)\big)}%
  {1 + (2-s-t)g_{\EN}(x) + (1-s)(1-t)(g_{\EN}(x)^2-g_{\EE}(x)g_{\NN}(x))}, \\
  G_{\NE}(x,s,t) &= \frac{g_{\EN}(x) + (1-t)\big(g_{\EN}(x)^2 - g_{\EE}(x)g_{\NN}(x)\big)}%
  {1 + (2-s-t)g_{\EN}(x) + (1-s)(1-t)(g_{\EN}(x)^2-g_{\EE}(x)g_{\NN}(x))}.
\end{align*}
\end{theorem}

\begin{proof}
First, split $G(x,s,t) = f(x) + \sum_{\ell=1}^\infty B_{\ell,0}(x) s^\ell + \sum_{r=1}^\infty B_{0,r}(x) t^r +
  \sum_{\ell,r\ge 1} B_{\ell,r}(x) s^\ell t^r$. Using the identities from equation \eqref{eq:onesidedBounces}, we get 
\begin{equation*}
 \sum_{\ell=1}^\infty B_{\ell,0}(x) s^\ell + \sum_{r=1}^\infty B_{0,r}(x) t^r  =
 \frac{sf_{\s\E}(x)f_{\N\s}(x)}{1- sf_{\EN}(x)} + \frac{tf_{\s\E}(x)f_{\N\s}(x)}{1- tf_{\EN}(x)}.
\end{equation*}
Furthermore, by the previous lemma and the basic identity
\begin{equation*}
  \sum_{\ell=j+1}^\infty \sum_{r=i}^\infty \binom{\ell-1}{j}\binom{r-1}{i-1} y^\ell z^r 
  = \Big(\frac{y}{1-y}\Big)^{j+1}\Big(\frac{z}{1-z}\Big)^{i},
\end{equation*}
the term $\sum_{\ell,r\ge 1} B_{\ell,r}(x) s^\ell t^r$ can be written as the sum of the following four functions: 
\begin{align*}
  & f_{\s\E}(x) f_{\N\s}(x) \frac{s}{1-sf_{\EN}(x)} \sum_{i=1}^\infty
  \left[ \frac{st f_{\EE}(x)f_{\NN}(x)}{(1-sf_{\EN}(x))(1-tf_{\EN}(x))} \right]^i, \\
  & f_{\s\E}(x)^2 \frac{1}{f_{\EE}(x)} \sum_{i=1}^\infty
  \left[ \frac{st f_{\EE}(x)f_{\NN}(x)}{(1-sf_{\EN}(x))(1-tf_{\EN}(x))} \right]^i, \\
  & f_{\s\N}(x)^2 \frac{1}{f_{\NN}(x)} \sum_{i=1}^\infty
  \left[ \frac{st f_{\EE}(x)f_{\NN}(x)}{(1-sf_{\EN}(x))(1-tf_{\EN}(x))} \right]^i, \\
  & f_{\s\E}(x) f_{\N\s}(x) \frac{t}{1-tf_{\EN}(x)} \sum_{i=2}^\infty
  \left[ \frac{st f_{\EE}(x)f_{\NN}(x)}{(1-sf_{\EN}(x))(1-tf_{\EN}(x))} \right]^{i-1}.
\end{align*}
Adding these functions and simplifying, we arrive at

\begin{equation*}
 G(x,s,t) = \frac{f - (s+t)\big(f_{\EN}^2-f_{\EE}f_{\NN}\big)}{(1-sf_{\EN})(1-tf_{\EN}) - st f_{\EE}f_{\NN}},
\end{equation*}
which by Theorem~\ref{thm:bounce_free} gives
\begin{equation*}
 G(x,s,t) = \frac{g + (2-s-t)\big(g_{\EN}^2 - g_{\EE}g_{\NN}\big)}%
 {1 + (2-s-t)g_{\EN} + (1-s)(1-t)\big(g_{\EN}^2-g_{\EE}g_{\NN}\big)}.
\end{equation*}
Finally, the claimed equations for the restricted functions $G_{\mathsf{ab}}(x,s,t)$ can be obtained in a similar way after making the appropriate adjustments to Lemma~\ref{lem:bounceStats}.
\end{proof}

Note that when $s=t=0$ we recover the expression for $f(x)$ (bounce-free case) from Corollary~\ref{cor:bounce_free}. Also, with $s=0$ and $t=1$ (or equivalently $s=1$ and $t=0$) we recover Lemma~\ref{lem:no_leftBounces}.

\section{The case when $\beta=1$}
\label{sec:beta=1}

For the special case when $\beta=1$, we have
\begin{equation*}
   g(x)=\sum_{k=1}^{\infty}\binom{(\alpha+1)k}{k}x^k,
\end{equation*}
which is related to the generating function for the Fuss-Catalan numbers $\frac{1}{\alpha k + 1} \binom{(\alpha+1)k}{k}$. Also,
\begin{equation*}
g_{\EE}(x)=\sum_{k=1}^{\infty}\tbinom{(\alpha+1)k-2}{k}x^k, \;\;
g_{\EN}(x)=\sum_{k=1}^{\infty}\tbinom{(\alpha+1)k-2}{k-1}x^k, \;\;
g_{\NN}(x)=\sum_{k=1}^{\infty}\tbinom{(\alpha+1)k-2}{k-2}x^k,
\end{equation*}
and equation \eqref{eq:gRelations} simplifies to
\begin{equation*}
  g_{\EE}(x) = \alpha g_{\NN}(x) + (\alpha -1)g_{\EN}(x).
\end{equation*}
Moreover, one can prove (combinatorially) the related identity
\begin{equation*}
  f_{\EE}(x) = f_{\NN}(x)+(\alpha-1)f_{\EN}(x),
\end{equation*}
and we get $g_{\EN}(x)^2 - g_{\EE}(x)g_{\NN}(x) = g_{\NN}(x)$. This allows us to simplify some of the identities from the previous sections. In fact, Theorem~\ref{thm:bounce_free} becomes
\begin{equation*}
 f_{\EE}(x) =\frac{g_{\EE}(x)}{1 + g(x) - g_{\EE}(x)}, \;\;
 f_{\EN}(x) =\frac{g_{\NN}(x) + g_{\EN}(x)}{1 + g(x) - g_{\EE}(x)}, \;\;
 f_{\NN}(x) =\frac{g_{\NN}(x)}{1 + g(x) - g_{\EE}(x)},
\end{equation*}
and the function $G(x,s,t)$ takes the form
\begin{equation*}
G(x,s,t) = \frac{g(x) + (2-s-t)g_{\NN}(x)}{1 + (2-s-t)g_{\EN}(x) + (1-s)(1-t)g_{\NN}(x)}.
\end{equation*}

Moreover, in terms of the function 
\[ c_{\alpha}(x) = \sum_{k=0}^\infty \frac{1}{\alpha k + 1} \binom{(\alpha+1)k}{k} x^k, \] 
the generating functions for the bounce-free $\mathsf{AB}$-paths can be written as
\begin{align*}
 f_{\EE}(x) &=\frac{(\alpha c_{\alpha}(x)-1)(c_{\alpha}(x)-1)}{(1-\alpha)c_{\alpha}(x)^2 + (\alpha+1)c_{\alpha}(x)-1}, 
 \\[1ex]
 f_{\NN}(x) &=\frac{(c_{\alpha}(x)-1)^2}{(1-\alpha)c_{\alpha}(x)^2 + (\alpha+1)c_{\alpha}(x)-1}, \\[1ex]
 f_{\EN}(x) &=\frac{c_{\alpha}(x)(c_{\alpha}(x)-1)}{(1-\alpha)c_{\alpha}(x)^2 + (\alpha+1)c_{\alpha}(x)-1}. 
\end{align*}

For $\alpha>1$ there seems to be no results in the literature that address the bounce statistics encoded by $G(x,s,t)$. Only for $\alpha=2$, the OEIS \cite[A000259, A000305]{OEIS} gives matches for the bounce-free functions
\begin{align*}
 f_{\EE}(x) &= x + 4x^2 + 18x^3 + 89x^4 + 466x^5 + 2537x^6 + 14209x^7 + 81316x^8 + \cdots, \\
 f_{\EN}(x) &= x + 3x^2 + 13x^3 + 63x^4 + 326x^5 + 1761x^6 + 9808x^7 + 55895x^8 +\cdots,
\end{align*}
which are connected to the enumeration of certain rooted planar maps, see \cite{Br63,Tu63}. We leave it as an open problem to find corresponding bijections.

\bigskip
If $\alpha=\beta=1$, we can use properties of the Catalan generating function $c(x) = c_1(x)$ to further simplify the functions above. In this case, we have
\begin{equation*}
 f_{\EE}(x) = f_{\NN}(x) = \frac{xc(x)^2 - xc(x)}{1+xc(x)}, 
 \;\; f_{\EN}(x) = \frac{xc(x)^2}{1+xc(x)}, \;\text{ and }\; f(x) = 2(c(x)-1).
\end{equation*}
Moreover,
\begin{equation*}
 G(x,s,t) = \frac{2xc(x) + (2-s-t)(xc(x) - x)}{1 - 2xc(x) + (2-s-t)x + (1-s)(1-t)(xc(x) - x)}.
\end{equation*}

For this special case ($\alpha=\beta=1$), our results coincide with the bounce statistics obtained by Pan and Remmel in \cite{PanRe16}. We refer to their work for other related statistics and information regarding some combinatorial applications. Here we would like to mention the interesting case of lattice paths having exactly $b$ bounces (left or right) for which the generating function becomes:\footnote{Using Lemma~\ref{lem:bounceStats} together with the additional simplification $f_{\s\E}(x) f_{\N\s}(x) = c(x) -1$.}
\begin{equation}\label{eq:bbouncesGF}
 G_b(x) = 2(c(x) - 1)^{b+1} = 2x^{b+1} c(x)^{2b+2} 
  = 2\sum_{k=1}^\infty \frac{b+1}{k+b}\binom{2k+2b}{k-1} x^{k+b}. 
\end{equation}
For $b=1,\dots,6$, these sequences are listed in the OEIS (A002057, A003517, A003518, A003519, A090749, A268446) with corresponding combinatorial interpretations, including the enumeration of lattice paths, permutation patterns, standard Young tableaux, and other interesting families.

An interesting application of \eqref{eq:bbouncesGF}, not discussed in \cite{PanRe16}, is the following connection between lattice paths with a prescribed number of bounces and standard Young tableaux.

\begin{corollary} \label{cor:SYT}
For $b,n\in\mathbb{N}$ and $n>b$, the set of lattice paths from $(0,0)$ to $(n,n)$ that start with an $\E$-step and have exactly $b$ bounces (left or right) is in one-to-one correspondence with the set of standard Young tableaux of shape $(n+b,n-b-1)$.
\end{corollary}

\subsection*{Horizontal crosses}
In the context of this paper, a {\em cross} on a lattice path $L\in \L_{\beta/\alpha}(k)$ is a point where $L$ traverses  the line $y=\frac{\beta}{\alpha}x$. In general, enumerating the elements of $\L_{\beta/\alpha}(k)$ by the number of crosses (vertical or horizontal) is more involved due to the fact that crosses may happen on a non-lattice point. However, for $\beta=1$, horizontal crosses (if any) occur on a lattice point and can be easily handled with an inclusion-exclusion argument as we did for the enumeration of bounces.

\begin{lemma} 
Let $\textup{nhc}_{\mathsf{ab}}(x)$ denote the generating function for the number of $\mathsf{AB}$-paths in $\L_{1/\alpha}(k)$ with no horizontal crosses. Then
\begin{equation*}
 \textup{nhc}_{\EE}(x) = \frac{g_{\EE}(x)}{1+g_{\EE}(x)} \;\text{ and }\; 
 \textup{nhc}_{\EN}(x) = \textup{nhc}_{\NE}(x) = \frac{g_{\EN}(x)}{1+g_{\EE}(x)}.
\end{equation*}
Thus, lattice paths in $\L_{1/\alpha}(k)$ that start with an $\E$-step and never cross the line $y=\frac{1}{\alpha}x$ horizontally are enumerated by the generating function
\begin{equation*}
  h_{\E\s}(x) = \frac{g_{\EE}(x)+g_{\EN}(x)}{1+g_{\EE}(x)} = 
  \sum_{k=1}^\infty \frac{\alpha(\alpha+2)}{(\alpha+1)k+1}\binom{(\alpha+1)k+1}{k-1} x^k.
\end{equation*}
\end{lemma}

Using the same technique, it is rather straightforward to combine the avoidance of horizontal crosses with the avoidance of left and right bounces.

\begin{lemma} \label{lem:noRBHC}
Let $\L^{\textup{nhc}}_{1/\alpha}(k)$ be the set of paths in $\L_{1/\alpha}(k)$ with no horizontal crosses. The generating function for the number of paths in $\L^{\textup{nhc}}_{1/\alpha}(k)$ that start with an $\E$-step and have no right bounces is given by
\begin{equation*}
 H_{\E\s}(x) =\frac{h_{\E\s}(x)}{1+\textup{nhc}_{\EN}(x)}=\frac{g_{\E\s}(x)}{1+g_{\E\s}(x)}=\alpha(c_{\alpha}(x)-1),
\end{equation*}
where $c_\alpha(x)$ is the Fuss-Catalan generating function.
\end{lemma}

Observe that any path in $\L^{\textup{nhc}}_{1/\alpha}(k)$ that starts with an $\N$-step will remain above the line $y=\frac{1}{\alpha}x$ and must necessarily end with an $\E$-step. This is precisely the set of rational Dyck paths with slope $1/\alpha$, which is enumerated by the Fuss-Catalan numbers. Thus $H_{\N\E}(x)=c_{\alpha}(x)-1$.

Once again, for $\alpha>1$ we did not find results concerning the enumeration of horizontal crosses. Only for $\alpha=2$ the OEIS contains the sequence A046646 corresponding to
\[ H_{\E\s}(x) = 2x + 6x^2 + 24x^3 + 110x^4 + 546x^5 + 2856x^6 + 15504x^7 + 86526x^8 +\cdots \]
This function enumerates certain rooted planar maps, see \cite{Br63}.

\section{Summary and final remarks}

The main focus of this paper has been the enumeration of lattice paths from $(0,0)$ to $(\alpha k,\beta k)$ by their bounces with respect to the line $y=\frac{\beta}{\alpha} x$. A complete solution is provided by the generating functions in Theorem~\ref{thm:mainGF}, all expressed in terms of the simpler binomial functions $g_{\EE}$, $g_{\EN}$, and $g_{\NN}$.

The special case $\alpha=\beta=1$ was recently studied by Pan and Remmel \cite{PanRe16} from the enumerative and combinatorial points of view. For this case, we give a connection between lattice paths with exactly $b$ bounces and standard Young tableaux of shape $(n+b,n-b-1)$, see Corollary~\ref{cor:SYT}. 

It is worth mentioning that for $b=2$ or $3$, the function given in \eqref{eq:bbouncesGF} is connected to certain permutation patterns. In fact, $x^{3} c(x)^{6}$ counts the number of permutations containing a single occurrence of an increasing subsequence of length three, cf.\ \cite{Noo96, Gud10}, and $x^{5} c(x)^{8}$ enumerates permutations with a single increasing subsequence of length three, with no two elements adjacent, cf.\ \cite{Gud10}.  

For other values of $\alpha$ and $\beta$ there seems to be very little information in the literature. Only for the case $\alpha=2$, $\beta=1$, we found connections between bounce-free paths and the enumeration of certain rooted planar maps, cf.\ \cite{Br63}. However, the fact that the case $\beta=1$ leads to the Fuss-Catalan numbers suggests connections with other combinatorial structures.


\end{document}